\documentclass[11pt,reqno]{amsart}

\usepackage[hmargin=4cm,vmargin=4cm]{geometry}

\usepackage{amscd,amssymb}
\usepackage{tikz}
\usepackage{tikz-3dplot}
\usepackage{tikz-cd}

\newtheorem {theorem}                   {Theorem}
\newtheorem* {theorem*}                   {Theorem}

\newtheorem {lemma}{Lemma}
\newtheorem {proposition}     {Proposition}
\newtheorem* {prop*}     {Proposition}

\newtheorem {corollary}      {Corollary}

\theoremstyle{definition}
\newtheorem {definition}[equation]{Definition}
\newtheorem {remark}        {Remark}

\newtheorem {example}         {Example}

\newcommand{\pr} {\smallskip\noindent{\bf Proof\,\,}}

\def\R{\mathbb{R}}

\newcommand{\pp}[2]{\frac{\partial#1}{\partial#2}}

\newcommand{\bigslant}[2]{{\raisebox{.2em}{$#1$}\left/\raisebox{-.2em}{$#2$}\right.}}

\newcommand{\tx}{\tilde {X}}
\newcommand{\ta}{\tilde {\alpha}}

\title{The topology of Bott integrable fluids}
\author{Robert Cardona}\address{ Robert Cardona, Institut de Recherche Math\'ematique Avanc\'ee, Universit\'e de Strasbourg, 7 rue Ren\'e Descartes. 67084 Strasbourg, France. \it{e-mail: robert.cardona@math.unistra.fr }
 }
 \thanks{The author acknowledges financial support from the Spanish Ministry of Economy and Competitiveness, through the Mar\'ia de Maeztu Programme for Units of Excellence in R\&D (MDM-2014-0445) via an FPI grant. The author is also supported by the AEI grant PID2019-103849GB-I00/AEI/10.13039/501100011033, and AGAUR grant 2017SGR932.}

\begin{document}

\begin{abstract}
We construct non-vanishing steady solutions to the Euler equations (for some metric) with analytic Bernoulli function in each three-manifold where they can exist: graph manifolds. Using the theory of integrable systems, any admissible Morse-Bott function can be realized as the Bernoulli function of some non-vanishing steady Euler flow. This can be interpreted as an inverse problem to Arnold's structure theorem and yields as a corollary the topological classification of such solutions. Finally, we prove that the topological obstruction holds without the non-vanishing assumption: steady Euler flows with a Morse-Bott Bernoulli function only exist on graph three-manifolds.
\end{abstract}

\maketitle

\section{Introduction}

The evolution of the velocity field $X$ of an ideal fluid on a Riemannian three-manifold $(M,g)$ is modeled by the Euler equations:
\begin{equation*}
\begin{cases}
\partial_t X + \nabla_X X &= -\nabla p,\\
\operatorname{div}X=0,
\end{cases}
\end{equation*}
where the differential operators are taken with respect to the metric $g$, and the scalar function $p \in C^{\infty}(M)$ denotes the inner pressure. Finding stationary solutions for a fixed metric can be challenging, and requires analytic tools for the study of partial differential equations \cite{EP}. Another approach, which is more flexible and permits the use of tools from dynamical systems and differential geometry, is to study flows that are stationary solutions for some metric. In other words, the ambient Riemannian metric is a new variable of the equations.

In the analysis of stationary solutions, Arnold's celebrated structure theorem \cite{Arn} inaugurated the modern field of topological hydrodynamics \cite{AKh}. For closed manifolds, it provides an almost complete description of the rigid behavior of the flow if the Bernoulli function, which depends on the pressure function, is non-constant and analytic (or $C^2$ Morse-Bott). Except on an analytic stratified subset of positive codimension, the manifold is trivially fibered by invariant tori where the flow is conjugate to a linear field. However, Arnold's theorem is an \textit{a posteriori} conclusion: it gives the structure of such solutions but says nothing about their existence. Although it is possible to construct some explicit examples, there is no general construction of steady flows under the hypotheses of Arnold's structure theorem. Natural questions in this context arise: given a function that matches Arnold's description, is it the Bernoulli function of some stationary Euler flow? This ``converse problem" to Arnold's theorem was already posed by Dennis Sullivan in a series of lectures at CUNY in 1994. In the same spirit, Peralta-Salas proposed in \cite{P} to analyze whether the ambient topology of the manifold can be an obstruction to the existence of integrable steady Euler flows. We will address these two problems in the non-degenerate case of Morse-Bott functions.

 With the assumption that the vector field $X$ is non-vanishing, it was shown in \cite{CV} that the ambient three-manifold cannot be arbitrary: it is necessarily a graph manifold. This obstruction is not inherent to stationary Euler flows \cite{EG0}, and hold in fact in a more general setting \cite{FM}: a non-singular flow with a stratified first integral only exists on a graph manifold. We will refer to a non-vanishing solution to the Euler equations for some metric and a non-constant analytic (or Morse-Bott) Bernoulli function as an \textbf{Arnold flow}. 

In the first part of this work, we show that any graph manifold admits an Arnold flow with an analytic Bernoulli function.
\begin{theorem}\label{main}
Any closed, oriented graph three-manifold $M$ admits a non-vanishing steady solution to the Euler equations for some metric and non-constant analytic Bernoulli function.
\end{theorem}
From the discussion above, we know that the converse also holds and we obtain an equivalence. We prove this using the standard decomposition of Seifert manifolds. However, using arguments from the proof of the theorem above, one can use another decomposition of graph manifolds developed by Fomenko et alli to obtain a wealthier source of examples of Arnold flows. A key point of the theory is to consider first integrals of Morse-Bott type, a common assumption in the theory of Hamiltonian integrable systems \cite{F1}. In particular, we can construct an Arnold flow realizing as Bernoulli function, up to diffeomorphism, any possible Morse-Bott function. In the language developed in \cite{F1} for Hamiltonian integrable systems, the topological characterization of a graph three-manifold and the foliation by level sets of a Bott function is given by a molecule with gluing matrices. Reformulated in a different language, we will say that a Morse-Bott function is ``admissible" if it is not topologically obstructed to be the first integral of a non-vanishing volume-preserving flow.

\begin{theorem}\label{ArnTop}
Given an admissible Morse-Bott function $B$ on a graph manifold $M$, there exists a metric and a steady solution to the Euler equations with Bernoulli function $B$ (modulo diffeomorphism of $M$). The induced Riemannian volume can be chosen to be any arbitrarily fixed volume form.
\end{theorem}
This theorem shows that when we allow the metric to vary, the apparent difficulty to construct fluid flows with such a rigid behavior is overcome, and one can realize flows in every possible topological configuration. An immediate corollary is that the invariants developed for Bott integrable systems give a topological classification of Arnold flows with Morse-Bott Bernoulli function. This classification can be compared with the classification of vorticity functions of Morse type studied in \cite{IK} for the 2D Euler equations in surfaces, which was further extended to surfaces with boundary \cite{K}.

 The resemblance between Arnold's structure theorem and Arnold-Liouville's theorem for integrable Hamiltonian systems already suggests that there might be some connection between these two worlds. By taking a careful look at the Arnold flows constructed in Theorem \ref{ArnTop}, we can relate them to Bott integrable systems through the symplectization of an appropriate stable Hamiltonian structure (confer Theorem \ref{maincor}). However, the construction is \emph{ad hoc} and we can find examples of Arnold flows that cannot be thought, together with their curl, as a Hamiltonian integrable system. This yields an alternative proof that any molecule with gluing matrices can be realized by an integrable system, with the additional property that the isoenergy hypersurface is of stable Hamiltonian type. Without this extra property, this was originally proved in \cite{BFM}.\\

 In the last part of this work, we drop the assumption that the solution is non-vanishing, and analyze in general steady Euler flows with a Morse-Bott Bernoulli function. Without the non-vanishing assumption, no obstruction on the topology of $M$ was previously known. We show that the Bernoulli function cannot have a non-degenerate critical point (independently of the Riemannian metric), and use it to generalize the obstructions to the existence of any steady solution with a Morse-Bott Bernoulli function. As mentioned earlier, this was only known with the assumption that the flow has no stagnation points.
 
 \begin{theorem}\label{thm:sing}
Let $M$ be a three-manifold that is not of graph type. Then $M$ does not admit a stationary solution to the Euler equations, for any metric, with a Morse-Bott Bernoulli function.
\end{theorem}  
For example, the class of hyperbolic manifolds satisfies the hypotheses of this theorem. This is the first example of a topological obstruction to the existence of an integrable fluid flow, and answers the question raised in \cite{P} in the Morse-Bott case.
 
 \textbf{Organization of this article.} In Section \ref{sec:pre} we recall the necessary definitions and results concerning the Euler equations in Riemannian manifolds and the construction of Seifert manifolds. In Section \ref{sec:arn}, we use the topological decomposition of Seifert manifolds to prove Theorem \ref{main}. In Section \ref{sec:fom}, we combine the techniques introduced in the previous section with the theory of Bott integrable systems developed by Fomenko et alli. We obtain both Theorem \ref{ArnTop} and \ref{maincor}, with the simplifying assumption that the Morse-Bott function has only simple atoms. In Section 5, we analyze Euler flows with a Morse-Bott Bernoulli function that can have stagnation points and obtain Theorem \ref{thm:sing}. In the Appendix, we cover the construction of Section $4$ for general $3$-atoms and discuss the topological classification of Arnold flows with Morse-Bott Bernoulli function.

 \section{Preliminaries}\label{sec:pre}

In this section, we introduce the basic background in the Euler equations and Seifert manifolds, required to prove Theorem \ref{main}. We will eventually introduce other preliminaries when required through the other sections.

\subsection{Steady Euler equations and structure theorem}
When the solution does not depend on time, we can rewrite the Euler equations in terms of the dual one-form $\alpha=g(X,\cdot)$ and the so-called Bernoulli function. The equations are then 
\begin{equation*}
\begin{cases}
\iota_X d\alpha=-dB, \\
d\iota_X\mu=0,
\end{cases}
\end{equation*}
where $B=p+\frac{1}{2}g(X,X)$ is the Bernoulli function and $\mu$ is the induced Riemannian volume. A key property of this function is that it is an integral of the field $X$. We aim to study vector fields that are solutions to such equations for \textit{some} metric. The vorticity field $Y$ of $X$, which also corresponds to the curl of $X$, is defined by the equation $\iota_Y\mu=d\alpha$.
\begin{definition}
Let $M$ be a three-dimensional manifold with a volume form $\mu$. A volume preserving vector field $X$ is \textbf{Eulerisable} if there is a metric $g$ on $M$ for which $X$ satisfies the Euler equations for some Bernoulli function $B:M \rightarrow \mathbb{R}$.
\end{definition}
We might refer to such vector fields as Euler flows too, and we note that the same definition makes sense on a manifold of any dimension (and the analogue of the Euler equations there). In \cite{PRT}, Eulerisable flows are characterized, and the following lemma is instrumental in such characterization.
\begin{lemma}\label{eul}
Let $X$ be a non-vanishing volume preserving vector field in $M$. Then there exists a one-form $\alpha$ such that $\alpha(X)>0$ and $\iota_Xd\alpha=-dB$ for some function $B\in C^\infty(M)$ if and only if $X$ satisfies the Euler equations for some metric $g$ satisfying $g(X,\cdot)=\alpha$.
\end{lemma}
When the Bernoulli function is assumed to be analytic, Arnold \cite{Arn} introduced a dichotomy depending on whether $B$ is constant or not. When the Bernoulli function is non-constant, the flow has a rigid structure that is reminiscent of integrable Hamiltonian systems in symplectic geometry, due to the fact that the velocity field and the vorticity field commute and are linearly independent almost everywhere. When the Bernoulli function is constant and Arnold's theorem does not apply, the solutions are Beltrami fields: vector fields parallel to their curl. 

The existence of non-vanishing Beltrami fields for some metric has been extensively studied. For instance, it was proved in \cite{EG} that solutions of this type exist in every homotopy class of non-vanishing vector fields (each solution for some particular metric) of any three-manifold. This result was recently extended to every odd-dimensional manifold \cite{C1}. Beltrami fields constitute a concrete class of stationary solutions, and their flexibility can be used to prove the existence of steady Euler flows with complex dynamical properties \cite{EP,CMPP}. The following statement is a version of Arnold's theorem in the context of closed manifolds. It holds also when the Bernoulli function is assumed to be $C^2$ Morse-Bott.
\begin{theorem}[Arnold's structure theorem]
Let $M$ be a closed three-manifold and $X$ a flow satisfying the Euler equations for some non-constant analytic Bernoulli function. Denote by $C$ the union of critical level sets of $B$, which is an analytic set of positive codimension. Then $M\setminus C$ consists of finitely many domains $M_i$ and each domain is trivially fibered by invariant tori of $X$, where the flow is conjugated to the linear flow.
\end{theorem}
In \cite{CV} the authors study the case of non-vanishing solutions in the hypotheses of the structure theorem. It is proved that for an analytic Bernoulli function $B$, one can always find some other metric in $M$ such that $X$ is a solution to the Euler equations with a constant Bernoulli function. In particular, it is the Reeb field of some stable Hamiltonian structure. As a corollary of the methods in the proof, topological obstructions to the existence of solutions with non-constant analytic Bernoulli function are obtained. The manifold has to be a union of Seifert manifolds glued along their torus boundary components, i.e. a graph manifold. This result was obtained in \cite{EG0} in a more general context: it holds for the existence of a non-vanishing vector field (not necessarily a steady Euler flow) with a stratified integral. Those conclusions follow from previous work in \cite{FM}. 
\begin{theorem}[{\cite[Corollary 2.10]{CV}}, {\cite[Theorem 5.1]{EG0}}]\label{top}
If a three-manifold admits a non-vanishing steady solution to the Euler equations for some metric and non-constant analytic (or in general stratified) Bernoulli function, then the manifold is of graph type.
\end{theorem}
 To simplify the notation, recall that we refer to a vector field which is a non-vanishing steady solution to the Euler equations for some metric and non-constant analytic Bernoulli function as an \textbf{Arnold flow}. This is reminiscent of the nomenclature introduced in \cite{KKP}, where one speaks of a divergence-free vector field which is \emph{Arnold integrable} when it is almost everywhere fibered by invariant tori.

\subsection{Seifert and graph manifolds}\label{Se}
Let us recall the definition and construction of Seifert manifolds, introduced and classified by Seifert \cite{Se}.

\begin{definition}
A \textbf{Seifert fiber space} is a three-manifold together with a decomposition as a disjoint union of circles. 
\end{definition}
Equivalently, a Seifert fiber space is given by a circle bundle over a two-dimensional orbifold. Denote by $\pi:M\rightarrow B$ the bundle map over the compact base $B$. When the manifold $M$ is oriented, Seifert fiberings are classified up to bundle equivalence by a set of invariants (up to some operations that yield isomorphic fiberings) of the following form:
$$ M= \{ g ; (\alpha_1,\beta_1),...,(\alpha_m,\beta_m)\}.$$
The only thing we will need in this work is the way to reconstruct the manifold $M$ from a given collection of invariants. The integer $g$ denotes the genus of the base space $B$, and $(\alpha_i,\beta_i)$ are pairs of relatively prime positive integers $0<\beta_i< \alpha_i$. The integer $g$ can be negative, and then $B$ is the connected sum of $g$ copies of $ \mathbb{RP}^2$. The integer $m$ represents the number of orbifold singularities. Given such a collection of invariants, there is a precise construction to obtain the manifold $M$ that we describe following \cite{JN}. 

The base space $B$ is either an orientable surface $\Sigma_g$ of genus $g$, or decomposes as $\Sigma_{g'}\# \mathbb{RP}^2$ or $\Sigma_{g'}\# \mathbb{RP}^2 \# \mathbb{RP}^2$ for some $g'$.  To simplify, denote in any case by $\Sigma_g$ the orientable surface with genus $g$ or $g'$ respectively in each case. Remove $m$ open disks of $\Sigma_g$, that we denote by $D_i, i=1,...,m$. If $B$ has  some non-orientable part, one or two additional open disks $\tilde D_1, \tilde D_2$ are removed depending on the amount of $\mathbb{RP}^2$ components that have to be summed to $\Sigma_g$ to obtain $B$. We will denote by $\Sigma_0$ the resulting surface with boundary.

Over $\Sigma_0$, consider the trivial $S^1$ bundle and denote it by $M_0= \Sigma_0 \times S^1$. At each component of the boundary, which is of the form $\partial D_i \times S^1$, we glue a solid torus $V_i= \overline{D^2} \times S^1$ by means of a Dehn surgery with coefficients $(\alpha_i, \beta_i)$. If $B$ is orientable, this concludes the construction: we obtain a closed three-manifold. If $B$ is non-orientable, we denote by $M_1$ the resulting three-manifold with boundary and we need to fill one or two boundary components $\partial \tilde D_1 \times S^1$ and $\partial \tilde D_2 \times S^1$ if there were respectively one or two copies of $\mathbb{RP}^2$ in the decomposition of $B$. To do so, consider the only orientable $S^1$-bundle over the M\"obius band: namely, the twisted $I$-bundle over the Klein bottle fibered meridianally. Denote two copies of such a space by $M_2$ and $M_3$. The torus boundary of both $M_2$ and $M_3$ can be framed longitudinally by a fiber and meridianally by a section to the bundle.  The boundary components $\partial \tilde D_1 \times S^1$ and $\partial \tilde D_2 \times S^1$ in $M_1$ can be framed also longitudinally by a fiber and meridianally by a section to the bundle. Then glue $M_2$ and $M_3$ respectively to the boundary components of $M_1$ according to the identity between the first homology classes represented by the given framings. This concludes the construction of $M$ in the most general case.

For a Seifert manifold with boundary, the base surface $B$ is a surface with boundary. The boundary of the Seifert manifold $M$ is then a collection of tori. 

\begin{definition}\label{def:graph}
A \textbf{graph manifold} is a manifold obtained by gluing Seifert manifolds with boundary along their torus boundary components.
\end{definition}
These manifolds were introduced and classified by Waldhausen \cite{W1,W2}. Observe that this is a larger class of manifolds since the gluing of the boundary tori might not match the fibering in each piece. Hence the resulting total space might not admit a foliation by circles.

\section{Arnold flows in Seifert and graph manifolds}
\label{sec:arn}
In this section, we will prove Theorem \ref{main}. To do so, we first show that any Seifert manifold $M$ admits an Arnold flow. Because of Lemma \ref{eul}, we only need to prove that there is vector field $X$ in $M$ and a one-form $\alpha$ such that
\begin{itemize}
\item $X$ is volume-preserving,
\item $\alpha(X)>0$,
\item $\iota_Xd\alpha=-dB$ for some analytical function $B$.
\end{itemize}
The strategy of the proof is to break the manifold $M$ into pieces according to the construction detailed above, construct a vector field and a one-form satisfying some conditions in each piece and glue them together in an appropriate way. Concretely, we will show that in a neighborhood of the gluing locus, one can interpolate between the Arnold flows in each piece to obtain a globally defined steady Euler flow. Once we have a globally defined pair $(X,\alpha)$, we show it is a steady Euler flow for some analytic Bernoulli function and hence an Arnold flow. Finally, we adapt the interpolation to glue Seifert manifolds along their boundary to deduce the general case of graph manifolds.

\subsection{Building pieces}\label{ssec:blocks}
Let $M$ be a three-dimensional manifold which is a Seifert fibered space. Hence $M$ is an $S^1$-fibration over a compact surface $B$. To cover the most general case, we shall assume that the base space decomposes as $\Sigma_g \# \mathbb{RP}^2 \# \mathbb{RP}^2$. As in Subsection \ref{Se}, the surface $\Sigma_0$ is an orientable surface of genus $g$ and obtained by removing some open disks $D_i, i=1,...,m$ and $\tilde D_i, i=1,2$ of $B$. It is a surface with boundary $M_0= \Sigma_0 \times S^1$. We think of $\Sigma_0$ as being embedded in $\R^3$ so that there is a naturally defined height $h$ function, whose minimum is equal to $1$ and is reached at the boundary components of $\Sigma_0$. This height function lifts trivially to $M_0$. See Figure \ref{fig:emb} for an example with genus $2$, two singular fibers and one $\mathbb{RP}^2$ component.
\begin{figure}[!ht] 

\begin{center}
\begin{tikzpicture}
     \node[anchor=south west,inner sep=0] at (0,0) {\includegraphics[scale=0.16]{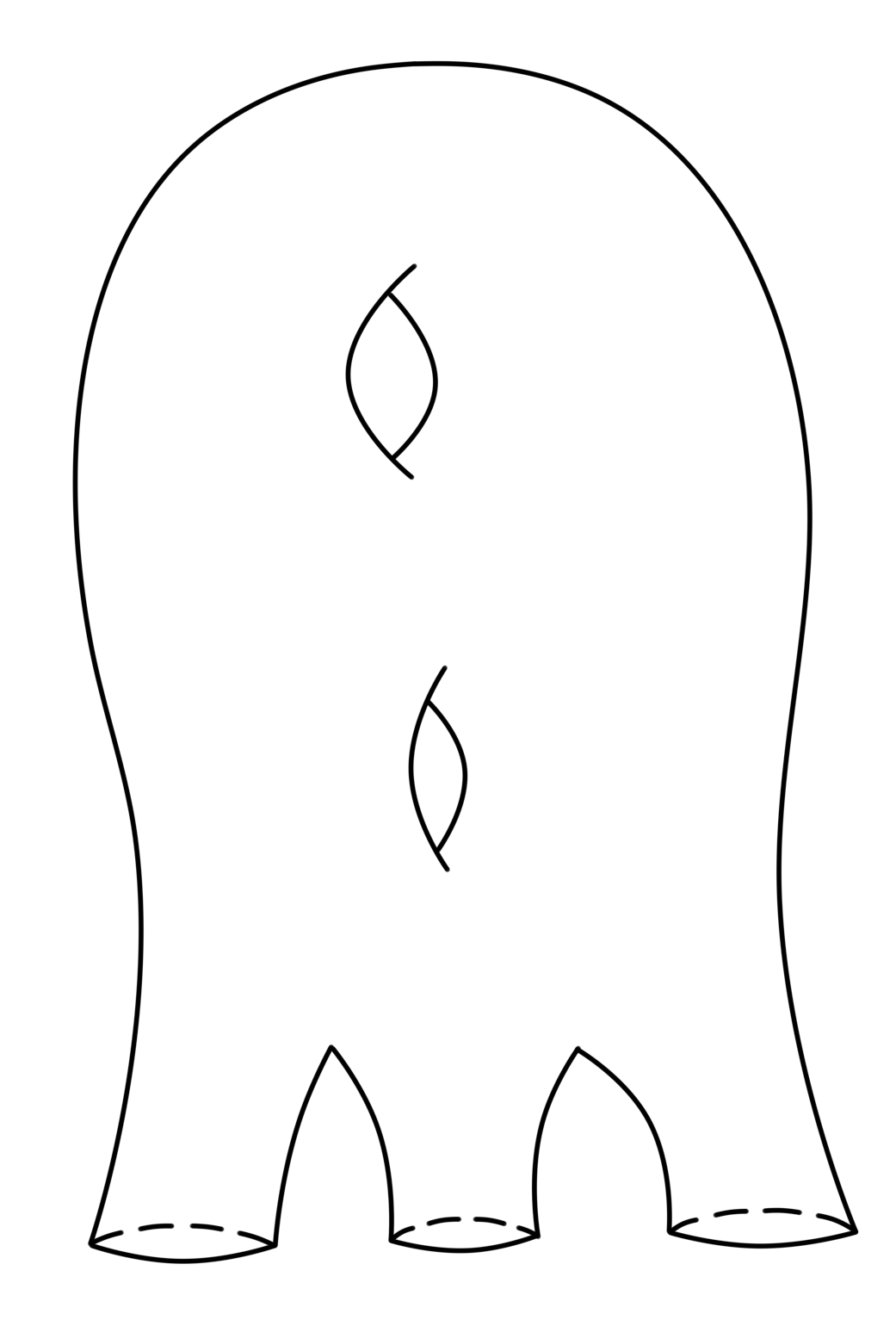}};
     
     \node at (4.5,5){$\Sigma_0$};
     \node at (0,0)[scale=0.8]{$\partial D_1$};
     \node at (2,0)[scale=0.8]{$\partial D_2$};
     \node at (4,0)[scale=0.8]{$\partial \tilde D_1$};
     \draw [->] (6,0) -- (6,6);
     \node at (6.5,3) {$h$};
     \node at (6.35,0.43) {$1$};
     \node at (6,0.43) {$-$};
\end{tikzpicture}
\caption{Height function in $\Sigma_0$}
\label{fig:emb}
\end{center}
\end{figure}

In $M_0$ denote by $\theta$ the coordinate in the $S^1$ component. Consider the vector field $ X= \pp{}{\theta}$, 
and the one-form $\alpha= hd\theta$. We have $\alpha(X)>0$ and $\iota_X d\alpha=-dh$.

On each solid torus $V_i \cong \bar D^2 \times S^1$, we take coordinates $(r,\varphi',\theta')$. Consider the vector field $Y= \pp{}{\theta'}$ and the one-form $\beta= v(r) d\theta'$, where $v'>0$ in $(0,1]$, close to $r=1$ the function is $r$ and close to $r=0$ the function is $\varepsilon + r^2$. We have $\beta(Y)>0$  and $\iota_Yd\beta= -v'(r)dr$. Hence taking as Bernoulli function $B=\int v'(r)dr$ such that $B(0)=0$, we have $\iota_Y d\beta=-dB$.

The remaining blocks $M_2$ and $M_3$ are twisted $I$-bundles over a Klein bottle. This space $M_2\cong M_3\cong K^2\tilde \times I$ can be seen as the mapping torus of
\begin{align*}
\phi: S^1\times [-1,1] &\longrightarrow S^1\times [-1,1] \\
 		(\theta,r)  &\longmapsto (-\theta, -r).
\end{align*} 
This means that there is a diffeomorphism
$$K^2 \tilde \times I \cong \bigslant{(S^1\times [-1,1])\times [0,1]}{ ((p,0)\sim (\phi(p),1))}.$$
This space is foliated by tori parallel to the boundary, and a copy of a Klein bottle at the core. If we denote by $r$ the coordinate in $[-1,1]$, the function $r^2$ is well defined in the quotient space. Furthermore, we have a naturally defined vector field $Z=\pp{}{\theta'}$, where $\theta'$ is in $[0,1]$, induced by the flow of the mapping torus and is $2$-periodic. Construct a function $v(r^2)$ such that close to $r=0$ it is equal to $\varepsilon + r^2$ and close to $r=1$ it is equal to $r^2$. By considering the one-form $\gamma=v(r)d\theta'$, we have $\gamma(Z)>0$ and $\iota_Y d\gamma =-dG$ where $G=\int v'(r)dr$ and $G(0)=0$. 

\subsection{Interpolation Lemma and gluing}

In order to reconstruct the whole manifold, we need first to glue back the sets $V_i= \overline{D_i}\times S^1$ to the boundary components $\partial D_i \times S^1$ of $M_0$ by means of a Dehn surgery operation. Let $U_i \cong [1,2) \times S^1 \times S^1 \subset M_0$ be a neighborhood of the boundary component $D_i \times S^1$. Fix $U$ to be one of these neighborhoods and endow it with coordinates $(t, \varphi, \theta)$. We can assume that the first component $t$ corresponds to the restriction of the height function $h$ to $U$, and that the coordinate $\theta$ is on the $S^1$ fiber of the trivial fibration in $M_0$. Denote by $\lambda= \{1\} \times \{\varphi_0\} \times S^1$ and by $\mu= \{1\} \times S^1 \times \{ \theta_0 \}$ the longitude and the meridian of the boundary component $\{1\} \times T^2$ of $U$. We will glue the solid torus $V= D^2 \times S^1$ (one of the $V_i$) to the boundary of $U$. Again take $\lambda_1= \{ \varphi'_0\in \partial {\bar D^2} \} \times S^1$ and $\mu_1= \partial {\bar D^2} \times \{\theta'_0\}$ the longitude and the meridian of $V$. The Dehn surgery is described by gluing in a way that 
\begin{align*}
\varphi:  \partial V_i &\longrightarrow \partial U_i \times S^1 \\
\mu_1 &\longrightarrow p\mu + q\lambda \\
\lambda_1 &\longrightarrow m \mu + n\lambda
\end{align*}
Coordinate wise we have $\varphi'=p\varphi+q\theta$ and $\theta'= m \varphi+ n\theta$. The coefficients $p,q$ are given by the invariants $\alpha_i, \beta_i$, and $m,n$ are any two integers such that $pn-qm=1$. We can assume that the surgery is such that the radial coordinate $r$ in $V$ is sent to $t$ in the neighborhood of the boundaries.  The vector field $Y=\pp{}{\theta'}$ which generates the longitude, is sent to $Y|_{\partial U}=  -q \pp{}{\varphi}+ p\pp{}{\theta}$ in the target coordinates, i.e. along the boundary of $U$. This can be easily deduced by the fact the coordinate wise we have 
\begin{equation*}
\begin{cases}
\varphi&= n \varphi' - q \theta' \\
\theta&= -m \varphi' + p\theta'
\end{cases}.
\end{equation*} 
Analogously, the one-form $\beta$ is sent to $\beta|_{\partial  U}= t(m d\varphi + nd\theta)$. This follows from the fact that near the boundary of $B$, the one-form is $\beta= r d\theta$. It satisfies $\beta(Y)=t(-qm+pn)=t$.

Once we have our building blocks of an Euler flow, we need to glue the flows in a smooth way: both the vector field and the one-form, and such that the critical set of the Bernoulli function is controlled. More precisely, we can do this interpolation by keeping the Bernoulli function regular.

\begin{lemma}[Interpolation Lemma]\label{int}
Suppose we are given a vector field $Y=A\pp{}{\theta}+B\pp{}{\varphi}$ in the torus $T^2$ and a one-form $\gamma=Cd\theta + D d\varphi$, for some constants $A,B,C,D$, such that $\gamma(Y)=AC+DB=1>0$. Denote $t$ the coordinate in $[1,2]$. Then there is a volume preserving vector field $X$ and a one-form $\alpha$ in $T^2\times [1,2]$ such that
\begin{itemize}
\item $X|_{\{t=1\}}=\pp{}{\theta}$, $\alpha|_{\{t=1\}}=d\theta$,
\item $X|_{\{t=2\}}=Y$, $\alpha|_{\{t=2\}}=\gamma$,
\item $\alpha(X)>0$ everywhere,
\item $\iota_X d(t\alpha)=-dh$ where $h(t)$ is a function without critical points and equal to $t$ at the boundary.
\end{itemize}
\end{lemma}
\begin{proof}
Break the interval $[1,2]$ into seven disjoint intervals $I_1,...,I_7$, for instance $I_i=[1+\frac{i-1}{7},1+\frac{i}{7}]$. Denote by $H_i(t)$ a cutoff function with support in $I_i$ such that $H_i'\geq 0$, with $H_i=0$ at $1+\frac{i-1}{7}$ and $H_i=1$ at $1+\frac{i}{7}$.

Since we have $AC+DB>0$, we might assume that $A$ and $C$ are of the same sign. Otherwise, the constants $D$ and $B$ are of the same sign and an analogous interpolation is done.

\begin{enumerate}
\item In the first interval, take $X=\pp{}{\theta}$ and $\alpha=d\theta+ H_1(t)d\varphi$. We have $\alpha(X)=1>0$.

\item In the second interval, take $X=\pp{}{\theta} + H_2(t)\pp{}{\varphi} $ and $\alpha=d\theta+ d\varphi$. We have $\alpha(X)=1+H_2(t)>0$.

\item In the third interval, take $X= (1-H_3(t))\pp{}{\theta}+ \pp{}{\varphi}$ and $\alpha=d\theta+d\varphi$. We have $\alpha(X)=(1-H_3(t))+1>0$.

\item In the fourth interval, take $X=\pp{}{\varphi}$ and $\alpha=(1+H_4(t)(C-1))d\theta + d\varphi$. We have $\alpha(X)=1>0$.
\item In the fifth interval, take $X=H_5(t)A\pp{}{\theta} + (1-H_5(t))\pp{}{\varphi}$ and $\alpha= Cd\theta + d\varphi$. We have $\alpha(X)= ACH_5(t)+(1-H_5(t))>0 $ since $AC>0$.
\item In the sixth interval, take $X=A\pp{}{\theta}$ and $\alpha=Cd\theta + (1+H_6(t)(D-1))d\varphi$. We have $\alpha(X)=AC>0$.
\item In the last interval, take $X=A\pp{}{\theta}+ H_7(t) B\pp{}{\varphi}$ and $\alpha= Cd\theta + Dd\varphi$. We have $\alpha(X)=AC+H_7(t)BD>0$. It is positive because if $BD$ is positive, then $AC+H_7(t) BD> AC>0$ and if $BD$ is negative then $AC+H_7(t)BD>AC+BD>0$.
\end{enumerate}
This a priori arbitrary way of interpolating by steps is done to achieve both that $h$ has no critical points and that it suits an application in the last part of this section. The key point to achieve that $h$ has no critical points is to avoid decreasing a component in $\alpha$ which is non-vanishing when evaluated at $X$.

We constructed $(X,\alpha)$ satisfying the first three conditions. To check the fourth condition, we will do it in each interval. If we compute $\iota_Xd(t\alpha)$ we obtain:
\begin{equation*}
\iota_Xd(t\alpha)=
\begin{cases}
- dt \qquad &t\in I_1 \\
-[1+ H_2 ]dt \qquad &t\in I_2\\
-[ 2-H_3 ]dt \qquad &t\in I_3\\
-dt \qquad &t\in I_4\\
-[ H_5AC+(1-H_5) ] dt \qquad &t\in I_5\\
- AC dt \qquad &t\in I_6\\
- [ AC+H_7BD ]dt\qquad &t\in I_7
\end{cases}
\end{equation*}
It is clear from the case-by-case description that the function $h$, which is the indefinite integral with respect to $t$ of the $dt$-term, has no critical points in $[1,2]$. At the boundary we have $\iota_Xd(t\alpha)=-dt$.

The vector field preserves the volume form $\mu=dt\wedge d\varphi \wedge d\theta$, since at any point the vector field is of the form $X=A(t) \pp{}{\theta} + B(t) \pp{}{\varphi}$ which implies that $\mathcal{L}_X \mu=d\iota_X\mu=0$.
\end{proof}
We can now apply this lemma in $U= [1,2] \times T^2$ to interpolate between $(Y,\beta)$ and $(X,\alpha)$. By Lemma \ref{int}, we can do it so that the function $h$ extends in the building blocks as the function $B$ and has no critical points in $U$. Hence we constructed a pair $(\tx,\ta)$ in $M_1$ with $\ta(\tx)>0$ and $\iota_{\tx}d\ta=-dB'$ for some function $B'\in C^\infty(M)$.

Finally, it only remains to glue $M_2$ and $M_3$ to the boundary components of $M_1$, i.e. $W_1= \partial \tilde D_1 \times S^1$ and $W_2= \partial \tilde D_2 \times S^1$. Take for example $M_2$, and $M_3$ is glued analogously. The vector field $Z$ is tangent to the leaves of the foliation (by torus and core Klein bottle) in $M_1$, and on the torus leaves $Z$ is linear and periodic. Hence at the surgered boundary, the vector field $Z$ and the one-form $\gamma$ are sent to $Z=C_1\pp{}{\theta}+C_2\pp{}{\varphi}$ and to $t\gamma=t(D_1d\theta+D_2\varphi)$ for some constants $C_1,C_2,D_1,D_2$ with $\gamma(Z)=C_1D_1+C_2D_2=1$.

We are under the hypotheses of Lemma \ref{int} in a neighborhood of the surgery locus, so we can construct a vector field and one-form $(X', \alpha')$ in $W_1$ that can be extended in $M_1$ as $(\tx,\ta)$ and in $M_2$ as $(Z, \gamma)$. Doing this to both pieces $M_2$ and $M_3$, we constructed a globally defined function $B\in C^\infty(M)$ such that $\alpha'(X')>0$, the equation $\iota_{X'} \alpha'=-dB$ satisfied and $B$ coincides with $\tilde B$ in $M_1 \setminus U_1 \cup U_2$ and with $G$ in $M_2$ and $M_3$. 

\subsection{Volume preservation and critical set}\label{ss:volcrit}

To prove that $X'$ is volume-preserving, we will prove that it preserves some volume in each part of the manifold that coincides along the glued boundaries.

\begin{enumerate}
\item Denote by $A$ a neighborhood of the boundary tori of $M_0$ where we applied the interpolation lemma. Then in $M_0\setminus A$ we have $X'=X=\pp{}{\theta}$ and $X$ clearly preserves $\mu= \mu_{\Sigma_0} \wedge d\theta$, where $\mu_{\Sigma_0}$ obtained by pulling back any area form of $\Sigma_0$ by the natural projection $\pi: \Sigma_0 \times S^1\rightarrow \Sigma_0$. This follows from the fact that $\iota_X\mu= \mu_{\Sigma_0}$, which is a closed form.
\item In any interpolation domain, the vector field is of the form $X'= A(t)\pp{}{\theta}+B(t)\pp{}{\varphi}$. Taking as volume $\mu_U= dt \wedge d\varphi \wedge d\theta$, we have 
$$ \iota_{\tx}\mu_U= A(t) dt \wedge d\varphi - B(t)dt\wedge d\theta, $$
 which is again closed. $\mu_U$ coincides with $\mu$ for a suitable choice of $\mu_{\Sigma_0}$.
 \item In every solid torus $V_i$, we have $Y=\pp{}{\theta'}$ which preserves the volume form $\phi(r)dr\wedge d\varphi'\wedge d\theta'$, for any function $\phi$ equal to $r^2$ near $r=0$. We might choose $\phi(r)$ to be constantly equal to $1$ near the boundary of $V_i$ so that it coincides with $\mu_U$.
 \item In the mapping tori $M_2$ and $M_3$, the vector field $X'=\pp{}{\theta'}$ preserves the volume form $\mu= \mu_S \wedge d\theta' $, where $\mu_S$ is the two form induced by any area-form $\phi(r^2)d\varphi \wedge dr$ of $S^1\times [-1,1]$. It is preserved by the diffeomorphism that we used to define the mapping torus, so it yields a two-form $\mu_S$ in $M_2$ or $M_3$. Choosing $\phi$ in an appropriate way, the volume form extends as $\mu_U$ in the gluing locus.
\end{enumerate}
In conclusion, we construct a globally defined volume $\mu$ on $M$ preserved by $X'$. We just proved that $X'$ is volume-preserving and admits a one-form $\alpha'$ such that $\alpha'(X')>0$ and $\iota_{X'} \alpha'=-dB$ for some function $B\in C^\infty(M)$. Applying Lemma \ref{eul}, we deduce that $X'$ satisfies the Euler equations for some metric with Bernoulli function $B$.\\

\paragraph{\textbf{Analyticity}}
The function $B$ is a priori only smooth. However, we have the following theorem of equivalence between smooth and analytic functions. We state a particular case that is enough for our purposes.

\begin{theorem}[{\cite[Theorem 7.1]{Sh}}] \label{An}
Let $f$ be a smooth function on a manifold $M$. Suppose that at every point we have locally that:
\begin{enumerate}
\item $f$ is regular,
\item $f$ is the sum of a constant and a power of a regular function,
\item $f$ is $\pm x_1^2 \pm ... \pm x_k^2 + \operatorname{const}$ for suitable coordinates $(x_1,...,x_n)$.
\end{enumerate}
Then $f$ is equivalent to an analytic function.
\end{theorem}

We only need to show that $B$ satisfies the conditions of this theorem, then $B$ is analytic taking some suitable charts. In $M_0\setminus A$, we have $B=h$. We initially defined $h$ in $\Sigma_0$ where it is a Morse function, hence $h$ is a Morse-Bott function in the considered three-dimensional space $M_0$. In each solid torus $V_i$ attached via Dehn surgery, we have $B=r^2$, and the only singularity is of Morse-Bott type (the critical core circle). In the pieces $M_1$ and $M_2$, the Bernoulli function has a Klein bottle as the singular set, since at the core of the mapping torus we have $B=r^2$. The singularity is again of Morse-Bott type. Finally, on the regions $U_i$ and $W_i$, we know that $B$ has no critical points.

Hence, the only singular points of $B$ admit an expression as required by the theorem above. We just proved that every orientable Seifert manifold admits a steady Euler flow with a Morse-Bott Bernoulli function, and by Theorem \ref{An} this function is equivalent to a non-constant analytic Bernoulli function. Observe that the same construction holds if the Seifert manifold has some boundary components.

\begin{theorem}\label{main2}
Every compact Seifert manifold (with or without boundary)admits a non-vanishing steady solution to the Euler equations (for some metric) with a non-constant analytic Bernoulli function. The same holds for a smooth Morse-Bott function.
\end{theorem}

To prove the general case of graph manifolds, it only remains to glue together Seifert manifolds with boundary and obtain globally defined Arnold flows.
%\begin{figure}[!h]
%\begin{center}
%\begin{tikzpicture}
%     \node[anchor=south west,inner sep=0] at (0,0) {\includegraphics[scale=0.16]{embedding.png}};
%\end{tikzpicture}
%\end{center}
%\end{figure}

\subsection{Proof of Theorem \ref{main}}

Let $M_1,M_2$ be two Seifert manifolds with boundary. We shall assume that there is only a single torus component in the boundary. The manifold $M_1$ is glued to $M_2$ by a diffeomorphism 

$$\varphi: \partial M_1 \longrightarrow \partial M_2,$$
 between both torus boundaries. By Theorem \ref{main2}, we can construct steady Euler flows with non-constant analytic Bernoulli function in both manifold $M_1,M_2$.  It follows from the construction in the previous section that we can assume that in the neighborhood of the boundaries $\partial M_i=U_i=T^2 \times [1,2]$ we have $X_i=\pp{}{\theta_i}$ and $\alpha=t_i d\theta_i$ where $(\theta_i,\varphi_i,t_i)$ are coordinates in $U_i$. We might assume that we glue $M_1$ and $M_2$ using a domain $U=T^2\times [-1-\varepsilon,1+\varepsilon]$ such that $U\cap U_1= T^2\times [-1,-1-\varepsilon]$ and $U\cap U_2=T^2\times [1,1+\varepsilon]$. That is, we thickened the gluing locus. We can find a coordinate $t$ in $U$ such that
\begin{align}\label{par}
&t^2=t_1 \text{ for } t\in [-1-\varepsilon,-1] \\
  &t^2=t_2 \text{ for } t\in [1,1+\varepsilon]\label{par2}
\end{align} 
   Consider the coordinates $(\theta_2,\varphi_2,t)$ of $U$, obtained by extending $\theta_2,\varphi_2$ from the boundary of $M_2$ to $U$.

The gluing diffeomorphism is by construction determined by some Dehn coefficients. The vector field $X_1$ and the one-form $\alpha_1$, which are in $U\cap U_1$ are of the form $X_1=C_1\pp{}{\theta_2} + C_2\pp{}{\varphi_2}$ and $\alpha_1=t^2( D_1d\theta_2 + D_2d\varphi_2)$ for some constants $C_1,C_2,D_1,D_2$. On the other hand, we have $X_2=\pp{}{\theta_2}$ and $\alpha_2=t^2d\theta_2$ defined in $U\cap U_2$. Using Lemma \ref{int} in $[-1-\varepsilon,-1]$, we construct a vector field $X'$ and a one-form $\beta$ satisfying $\iota_{X'}\beta=-dH$ for some function $H$ without critical points that extends as $t^2$ in all $U$. Furthermore, in $T^2 \times \{-1\}$, the pair is equal to $(X_1,\alpha_1)$ and in $T^2 \times \{1\}$ is equal to $(\pp{}{\theta},t^2d\theta)$. Finally, the vector field $\pp{}{\theta}$ and the one-form $t^2d\theta$ defined in $U$ extend everywhere in $M_1$ and $M_2$ by conditions \eqref{par} and \eqref{par2}. 

We obtained a globally defined volume-preserving vector field $X'$ and a one-form $\beta$ such that $\beta(X)>0$ and $\iota_X \beta=-dB$. The only new critical level set of the Bernoulli function is given by $T^2 \times \{t=0\}$, a non degenerate critical torus, since there the Bernoulli function is equal to $t^2$. This concludes the proof of the Theorem \ref{main}, since the real analyticity of $B$ follows again from Theorem \ref{An}.

\section{Fomenko's theory for Arnold flows}\label{sec:fom}

In this section, we use the theory of Bott integrable systems developed by Fomenko et alli to realize Arnold flows with any possible Morse-Bott Bernoulli function. The connection between these steady flows and the theory of Morse-Bott integrable systems was already used in \cite{EG0}. We will show that any topological configuration (in the sense of a graph manifold and a given admissible Morse-Bott function) can be realized by an Arnold flow. To simplify the discussion, we will treat in this section the case where the Morse-Bott integral contains a single connected critical submanifold in each connected component of the critical level set. In the language of atoms introduced in \cite{BF}, this means that we assume that the Bott integral only has simple atoms. We leave for the Appendix the case of arbitrary $3$-atoms, where we also discuss the topological classification of the moduli of Arnold flows with Morse-Bott Bernoulli function. Taking into account this appendix, we get a proof of Theorem \ref{main2}.

\subsection{Topology of Bott integrable systems}
\label{ssec:blocks2}

We will describe in this subsection some aspects of the topological classification of integrable systems with Bott integrals in isoenergy surfaces of dimension three. For more details on this theory, we confer the reader to \cite{F1} which we will mainly follow in the discussion.

Consider a symplectic manifold $(M,\omega)$ of dimension four with equipped with an integrable Hamiltonian system $F=(H,f)$, where $H,f\in C^\infty(M)$. This means that the pair of functions satisfies $dH\wedge df \neq 0$ almost everywhere and that they commute with respect to the Poisson bracket $\{f,H\}=0$. Denote by $Q$ a three-dimensional regular isoenergy level set of $H$, and assume that $f$ restricts to $Q$ as a Morse-Bott function.

Let us denote by $(H)$ the class of orientable closed three-manifolds that are isoenergy hypersurfaces of some integrable Hamiltonian system with the properties described above, in some four-dimensional symplectic manifold with boundary. Similarly, we denote by $(G)$ the class of orientable graph manifolds, as introduced in Definition \ref{def:graph}, and $(Q)$ the class of three-manifolds that can be decomposed into the sum of ``elementary bricks" which are solid tori $D^2\times S^1$, a torus times an interval $T^2\times I$ or $N^2 \times S^1$. Here $N^2$ denotes a disk with two holes. In a series of papers \cite{BF,F2,FZ1,FZ2}, it was proved that all three classes coincide.
\begin{theorem}[Brailov-Fomenko, Fomenko, Fomenko-Zieschang]\label{FZ}
The three classes coincide, i.e. we have $(H)=(Q)=(G)$.
\end{theorem}

With the assumption that we took on the critical level set of $f$ (i.e. that there are only simple atoms), up to five types of blocks describe the topology of the foliation induced by the Bott integral. These five blocks are:
\begin{itemize}
\item Type $I$: The solid torus $S^1 \times D^2$.
\item Type $II$: The thick torus $T^2 \times [1,2]$.
\item Type $III$: The domain $N^2 \times S^1$, where $N^2$ is a $2$-dimensional disk with two holes.
\item Type $IV$: The mapping torus of $N^2$, with a rotation of angle $\pi$ that we will denote $N^2 \tilde \times S^1$.
\item Type $V$: The mapping torus of $S^1 \times [-1,1]$ by the diffeomorphism $\varphi(\theta,t)=(-\theta,-t)$. 
\end{itemize}

Denote by $f$ the Bott integral and consider the following integers counting the different critical submanifolds of $f$: the number $m$ of stable periodic orbits (minimum or maximum), the number $p$ of critical tori (minimum or maximum), the number $q$ of unstable critical circles with orientable separatrix diagram, the number $s$ of unstable critical circles with non-orientable separatrix diagram, and the number $r$ of critical Klein bottles (minimum or maximum). Then the manifold $M$ can be represented as $M=mI+pII+qIII+sIV+rV$, gluing the elementary blocks by certain diffeomorphisms of the torus boundary components.

If we further indicate how the blocks are connected by means of edges (if we want, oriented with respect to the direction in which the function increases its value), we obtain a complete topological description of the level sets of the Bott integral.

The topology of the function is determined by some graph, see also \cite{Cas}. The graph satisfies that each vertex has one, two or three edges. This holds because blocks of type $I$ and $V$ have one boundary component, blocks of type $II$ and $IV$ have two boundary components and blocks of type $III$ has three boundary tori. In order to fix the topology of the ambient manifold, one needs to specify the mapping class of each gluing diffeomorphism: the coefficient the Dehn twist. As described in \cite[Section 4.1]{F1}, there is a family of natural choices of framings (all equivalent) in each boundary torus, and hence the coefficient of a Dehn twist determines the gluing isotopy class. Fixing a graph with Dehn coefficients in each edge determines both the topology of the manifold and the topology of the function.

The theory of these graph invariants also applies to Arnold flows with a Morse-Bott Bernoulli function. A way to formalize it is the following lemma.

\begin{lemma}
Let $B$ denote the Bernoulli function of a non-vanishing steady Euler flow $X$ on a Riemannian three-manifold $(M,g)$. Then there is a symplectic form $\omega$ in $M\times [-\varepsilon,\varepsilon]$ such that $X$ is the Hamiltonian vector field of the coordinate $t$ in $[-\varepsilon,\varepsilon]$ along the zero level set. In particular if $dB\neq 0$ on a dense subset of $M$, we obtain that $(t,B)$ defines an integrable system on $(M\times [-\varepsilon,\varepsilon],\omega)$.
\end{lemma}

\begin{proof}
By assumption, the vector field $X$ satisfies
\begin{equation*}
\begin{cases}
\iota_Xd\alpha=-dB\\
d\iota_X\mu=0
\end{cases}
\end{equation*}
where $\alpha=g(X,\cdot)$ and $\mu$ is the Riemannian volume form. Denote by $t$ the coodinate in the second component of $M\times [-\varepsilon,\varepsilon]$, where $\varepsilon$ will be taken small enough. Consider the one-form $\beta=\frac{\alpha}{\alpha(X)}$ and construct the two-form
$$ \omega= d(t\beta) + \iota_X\mu. $$
It is closed since $d\omega=0+d\iota_X\mu=0$. On the other hand
\begin{align*}
\omega^2=dt\wedge\beta \wedge \iota_X\mu + td\beta\wedge \mu.
\end{align*}
But $dt\wedge \beta \wedge \iota_X\mu$ is a volume form, which implies that for $t$ small enough $\omega$ is non degenerate and hence a symplectic form in $M\times [-\varepsilon,\varepsilon]$. The vector field $X$ (trivially extended to $M\times [-\varepsilon,\varepsilon]$) satisfies $\iota_X\omega= -dt$ and so is the Hamiltonian vector field of $t$. This shows that $X$ on $M \times \{0\}$ is the restriction of a Hamiltonian vector field to a regular energy level set. Furthermore, since $B$ is an integral of $X$, we deduce that if $dB\neq 0$ on a dense set of $M$ we have $dt\wedge dB\neq 0$ on a dense set of $M\times [-\varepsilon,\varepsilon]$ and the pair $(t,B)$ defines an integrable system.
\end{proof}

In particular, if $B$ is Morse-Bott, we deduce that it is the integral of the Hamiltonian vector field $X$. Hence the pair $(M,B)$ is topologically classified by the theory of Bott integrable systems.

\begin{example}
Take for example the Arnold flow constructed in Theorem \ref{main} for Figure \ref{fig:emb}. Asssume we take a height function that only has a critical point in each value: this is true if the critical value joining $\partial D_2$ and $\partial \tilde D_1$ is lower than the critical value joining $\partial D_1$ and $\partial D_2$. A representation of the graph associated to such topological decomposition would be Figure \ref{fig:graph}. We took a framing in the boundary of the Klein bottle neighborhood for which the gluing is trivial as described in Section \ref{sec:pre}. Whenever the Dehn coefficients are trivial in some gluing, nothing is indicated in the edge. The coefficients $(\alpha_i,\beta_i)$ are indicated by the Seifert invariants.

\begin{figure}[!ht]
\begin{center}
\begin{tikzpicture}
     \node[scale=0.8] at (-0.8,-2,0) {$I$};
     \node[scale=0.8] at (1,-2,0) {$I$};
     \node[scale=0.8] at (2.8,-2,0) {$V$};
     
     \node[scale=0.8] at (1.9,-1) {$III$};
     \node[scale=0.8] at (1,0) {$III$};
     
     \node[scale=0.8] at (1,1,0) {$III$};
     \node[scale=0.8] at (1,2,0) {$III$};
     \node[scale=0.8] at (1,3,0) {$III$};
     \node[scale=0.8] at (1,4,0) {$III$};
     \node[scale=0.8] at (1,5,0) {$I$};
     
     \draw (-0.7,-1.8) -- (0.8,-0.2);
     \draw (1,0.2) -- (1,0.8);
     \draw (2.7,-1.8) -- (2,-1.2);
	 \draw (1.1,-1.8) -- (1.8,-1.2);   
	 \draw (1.8,-0.8) -- (1.1,-0.2);

     \draw (1.1,1.2) arc(-50:50:0.4);
     \draw (0.9,1.8) arc(130:230:0.4);
     
     \draw (1,2.2)--(1,2.8);
     
     \draw (1.1,3.2) arc(-50:50:0.4);
     \draw (0.9,3.8) arc(130:230:0.4);
     
     \draw (1,4.2) -- (1,4.8);
     
     \node[scale=0.7] at (-0.5,-0.9) {$(\alpha_1,\beta_1)$};
     \node[scale=0.7] at (0.85,-1.45) {$(\alpha_2,\beta_2)$};     
     
     \draw [->] (4,-2) -- (4,5);
     \node at (5,1.5) {$B$};
     
     \node at (-3,0) {};
\end{tikzpicture}
\end{center}
\caption{Example of graph representation}
\label{fig:graph}
\end{figure}

\end{example}

\subsection{Topological realization of Arnold flows}

Using similar arguments as we did to prove Theorem \ref{main}, we can construct an Arnold flow realizing each topogical configuration. Let us start by constructing an Arnold flow in each of the ``elementary blocks".

\begin{proposition} \label{ArnBl}
All blocks admit an Arnold flow with the following properties. For type $I$, the longitudinal core circle is a minimum or maximum of the Bernoulli function. For type $II$, the torus $T^2 \times \{3/2\}$ is a minimum or maximum of the Bernoulli function. For type $III$, critical set is a figure eight times a circle: the central circle is of saddle type. For type $IV$, exactly as for type $III$ but with a non orientable separatrix diagram for the critical circle. For type $V$, the core Klein bottle is a minimum or maximum of the Bernoulli function. In all cases, the boundary components are regular level sets of the Bernoulli function. For blocks $III$ and $IV$, we can assume that the Bernoulli function decreases (or increases) outwards in the exterior boundary component of $N^2$ and respectively increases (or decreases) in the other boundary components.
\end{proposition}

\begin{proof}
We construct in each block a vector field $X$ and a one-form $\alpha$ such that $\iota_X d\alpha=-dh$ for some function $h$ satisfying the mentioned properties. In all cases, it is easy to check that the vector field is volume-preserving as in Subsection \ref{ss:volcrit}.

In blocks of type $I$ and $V$, the construction is done as in Section \ref{sec:arn}. In the first one, the vector field is the longitudinal flow $\pp{}{\theta}$ with one-form $v(r)d\theta$, where $\theta$ is the longitudinal coordinate of the solid torus and $r$ the radial coordinate in $D^2$.  The function $v(r)$ is equal to $(\varepsilon + r^2)$ close to $r=0$ and equal to $ r$ close to the boundary $\{r=1\}$ if it is a minimum. If it a maximum we can take for example $v(r)=1+\varepsilon -r^2$ close to $r=0$ and $v(r)=1+\varepsilon - r$ close to $r=1$.

A type $V$ block is given by the mapping torus with core Klein bottle introduced in Subsection \ref{ssec:blocks}. The vector field is given by the mapping torus direction $\pp{}{\theta}$, and again the one-form is $v(r^2)d\theta$, where the function $v$ is equal to $\varepsilon + r^2$ close to $r=0$ and $r^2$ close to $r=1$. Similarly for a maximum, take $1+\varepsilon -r^2$ close to $r=0$ and $1+\varepsilon - r^2$ close to $r=1$. Recall that $r$ is the coordinate in $[-1,1]$, where the mapping torus is obtained by a diffeomorphism $\varphi: S^1 \times [-1,1] \rightarrow S^1 \times [-1,1]$.

For the type $II$ block, consider the standard coordinates $(\theta, \varphi,t)$ in $T^2 \times [1,2]$. Take the vector field $X=\pp{}{\theta}$ and the one-form $\alpha=v(t)d\theta$. If we want a minimum, we choose as function $v(t)=\varepsilon + (t-3/2)^2$, which implies that the Bernoulli function is $h=\int_t v'(t)= t^2-3t $. If we want it as a maximum then $v(t)=1+\varepsilon-(t-3/2)^2$ and $h=-t^2+3t$. We have $\iota_Xd\alpha=-dh$ and $\alpha(X)>0$.

For the type $III$ and type $IV$, denote by $\theta$ the coordinate in the $S^1$ component. Take a one-form $\alpha=v(x,y)d\theta$, where $(x,y)$ are coordinates in $N^2$. Taking the function $v$ such that it has a saddle point in $(0,0)$ and two minima or maxima in $(\pm 1,0)$ is enough. For example, we might choose
$$v(x,y)=K \pm \frac{1}{4}(y^2-x^2+1/2x^4),$$ 
for some big enough constant $K>0$ added to the Hamiltonian of the Duffing equation. Then as Bernoulli function we can choose similarly $h(x,y)= C \pm \frac{1}{4}(y^2-x^2+1/2x^4)$, for some constant $C$. We have
$$\iota_Xd\alpha=\iota_X (\pp{v}{x}dx\wedge d\theta + \pp{v}{y}dy\wedge d\theta)=-\pp{v}{x}dx-\pp{v}{y}dy,$$
 which is exactly equal to $-dh$. It is also satisfied that $\alpha(X)>0$. Depending on the sign, we obtain that $h$ is decreasing or increasing (outwards with respect to the boundary, that we take to be a level set of $h$) in the interior boundary components, and respectively increasing or decreasing in the exterior boundary components. Observe that the defined one-form and function $h$ are well defined in the mapping torus in the case of type $IV$ blocks. This is because the Duffing potential is invariant with respect to the rotation of angle $\pi$, which is easily seen in polar coordinates. So it yields a Morse-Bott function $\tilde h$ defined in the mapping torus. This covers the case of block $IV$.

The $N^2$ copy that we take is the one given by the level sets of the function $h$ or $v$: i.e. the boundary and holes we take are given by some of the regular level sets of these functions. Figure \ref{fig:level} gives a representation of the critical level set and the boundary level sets given by the function.
\begin{figure}
\begin{center}
\begin{tikzpicture}
     \node[anchor=south west,inner sep=0] at (0,0) {\includegraphics[scale=0.16]{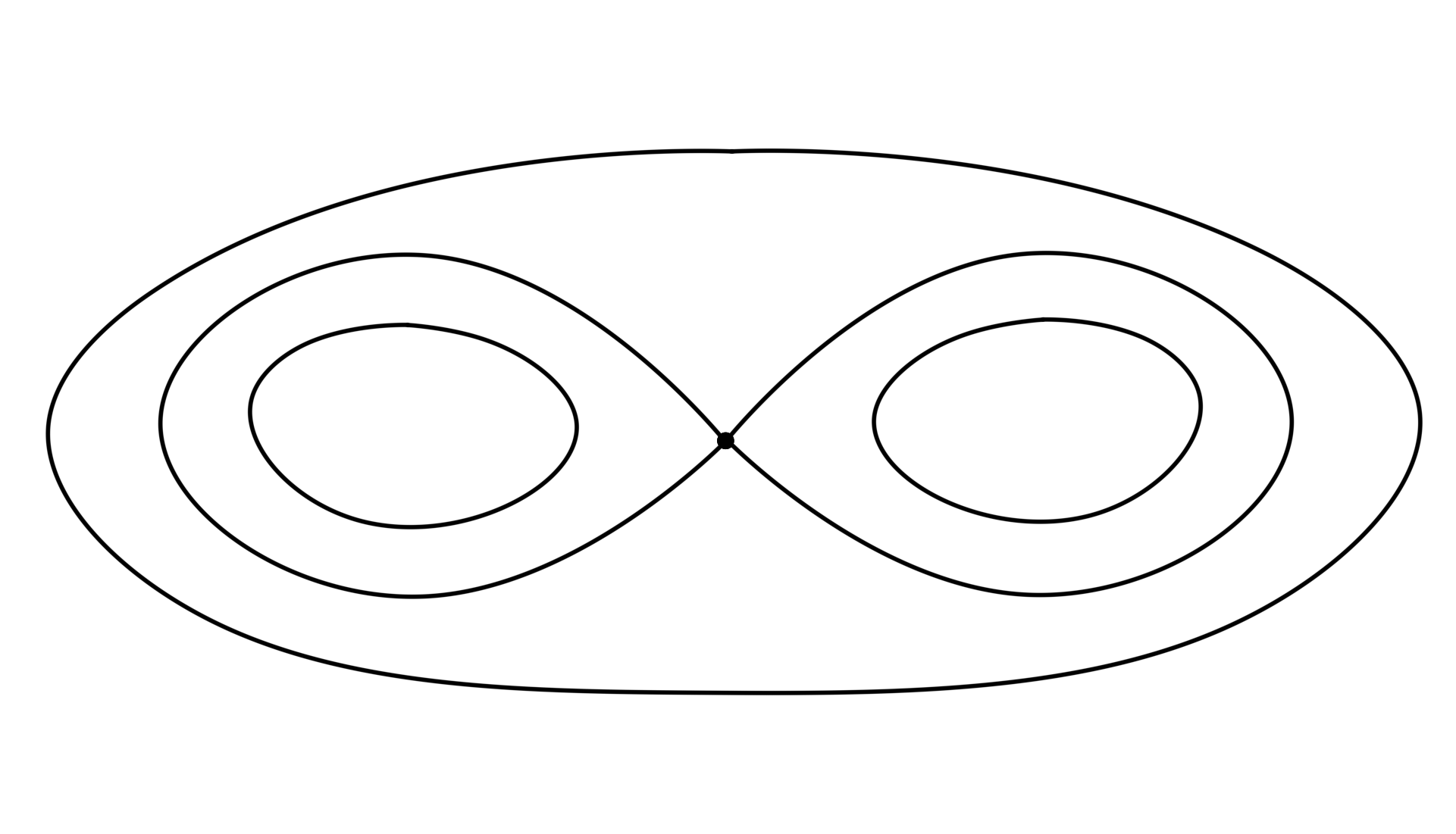}};
\end{tikzpicture}
\end{center}

\caption{Level sets of $h$}
\label{fig:level}
\end{figure}
 
 In all cases, the vector field is volume preserving and Lemma \ref{eul} concludes.
\end{proof}

Combining this result with the interpolation lemma, we can realize any configuration graph as Figure \ref{fig:graph}, and so any possible topological configuration is realized by an Arnold flow. We also state it in the general case of a molecule with gluing coefficients. 
\begin{theorem}\label{Mole}
Given a graph with blocks $I-V$ and Dehn coefficients, there exist an Arnold flow whose Morse-Bott Bernoulli function realizes it. In general, given a molecule with gluing coefficients there is an Arnold flow realizing it.
\end{theorem}
\begin{proof}
As in the whole section, we restrict to graphs with simple atoms, i.e. blocks of the form $I-V$, and leave for the appendix the general case.

Take a graph with Dehn coefficients and oriented edges. Each vertex indicates the type of block (and hence of the neighborhood of some connected component of a critical level set) of the Bernoulli function. The amount of up-directed edges for type $III$ blocks indicates if in the interior boundary components or in the exterior one the Bernoulli function is increasing. We start from the bottom and construct in the minima blocks an Arnold flow using Proposition \ref{ArnBl}. We proceed by induction.

Assume we have an Arnold flow in a manifold with boundary $N$ realizing a subgraph of the given marked molecule. Denote by $B$ the Bernoulli function in $N$. In a neighborhood of a torus boundary component of some of its blocks, there exist coordinates $(t,\theta)$ such that the one-form is $\alpha=td\theta$. We attach the following block, that we assume to be of type $III$ or $IV$, via a Dehn twist with the coefficients indicated by the edge of the graph. Using Proposition \ref{ArnBl}, we endow the block with an Arnold flow. Up to choosing well the constants $K$ and $C$ in Proposition \ref{ArnBl}, we can make sure that the minimal value of the Bernoulli function is higher than the maximal value of $B$ in $N$. Denote by $A$ and $B$ the maximal value of $B$ in $N$ and the minimal value of the Bernoulli function in the new block. Hence in a neighborhood of the gluing locus $U(T^2) \cong T^2 \times [A,B]$ we can assume that in each boundary component $T^2 \times \{A\}$ and $T^2 \times \{B\}$ there are respectively defined Arnold flows coming from $N$ and the glued block. We are in the hypotheses of the interpolation Lemma \ref{int}, since the vector fields are always linear in the torus boundaries. In a neighborhood $U(T^2)$, we obtain a globally defined non-vanishing vector field $X$, and a one-form $\alpha$ such that $\alpha(X)>0$ and $\iota_X d\alpha=-dB'$. Here $B'$ is a function which coincides with $B$ in $N$ except at the neighborhood where we applied the interpolation lemma. 

The cases of attaching the last blocks of type $I$, $II$ or $V$ containing a maximum of the Bott function are done analogously.

After iterating this process with all the edges of the graph, we obtain on a closed manifold a globally defined vector field, one-form, and Bernoulli function realizing the given graph. One easily checks that $X$ globally preserves some volume form as in \ref{ss:volcrit}. This proves, by Lemma \ref{eul}, there is some Arnold flow realizing the initial graph with coefficients.

In the general case, when we can have more than one critical circle in the same connected component of the critical level set, the neighborhood of a singular leaf is a $3$-atom as described in \cite{F1}. In the Appendix we explain how to construct an Arnold flow in the case of an arbitrary $3$-atom, in the sense of any possible foliation around a singular leaf. Then using the interpolation lemma as above proves that given any graph with arbitrary $3$-atoms of any complexity as vertices, there is an Arnold flow realizing it. 
\end{proof} 

\begin{remark}
By the same arguments as in the previous section, we can in fact construct a Bernoulli function that is analytic.
\end{remark}

Theorem \ref{ArnTop} stated in the introduction is just a reformulation of the realization of any marked molecule. However, in the previous Theorem we didn't fix a volume form a priori. A simple application of Moser's path method makes it possible to prescribed the volume form. 

\begin{lemma}
Let $X$ be a steady Euler flow with Bernoulli function $B$ on $(M,g)$. Let $\mu'$ be any volume form. Then there is some other metric $g'$ such that $X$ is a steady solution to the Euler equations with Bernoulli function $B$ (up to diffeomorphism) and induced Riemannian volume form $\mu'$.
\end{lemma}
\begin{proof}
Denote by $\mu$ the volume induced by $g$, which is preserved by $X$.  Up to multiplying $\mu$ by a constant, Moser's path method shows that there is a diffeomorphism $\varphi:M \rightarrow M$ (actually an isotopy) such that $\varphi^*\mu=\mu'$. In particular, the vector field $Y=\varphi^*X$ preserves $\varphi^*\mu=\mu'$. On the other hand, if $\alpha$ denotes $g(X,\cdot)$, we know it satisfies $\iota_Xd\alpha=-dB$. We deduce that $\beta=\varphi^*\alpha$ satisfies
$$ \iota_{\varphi^*X}d\varphi^*\alpha=-d\varphi^*B. $$
To conclude, we construct a metric such that $g'(Y,\cdot)=\beta$ and $\mu'$ is the induced Riemannian volume using Lemma \ref{eul}. We conclude that $Y$ satisfies the stationary Euler equations in $(M,g')$ with Bernoulli function $\varphi^*B$ and induced volume form $\mu'$.
\end{proof}
\subsection{Arnold flows as integrable systems}

In this last subsection, we will show how the Arnold flows previously constructed, together with their vorticity vector fields, can be interpreted as an integrable Hamiltonian system. 

\begin{theorem}\label{maincor}
The constructed steady Euler flows together with their vorticity in Theorem \ref{ArnTop} can be realized as the isoenergy hypersurface of a Hamiltonian system with a Bott integral (the Bernoulli function) in a symplectic manifold with boundary. Additionally, the Hamiltonian vector field is, up to rescaling, the Reeb field of a stable Hamiltonian structure.
\end{theorem}

We will need a one-form with some properties. Recall that a stable Hamiltonian structure is a pair $(\alpha,\omega)$ of a one-form and a two form such that $\alpha\wedge \omega>0$ and $\ker \omega \subset \ker d\alpha$. The equations $\iota_R \omega=0$ and $\alpha(R)=1$ uniquely define the Reeb field of the stable Hamiltonian structure. It was already proved in \cite{CV} that any non-vanishing steady solution to the Euler equations with non-constant analytic Bernoulli function can be rescaled to the Reeb field of a stable Hamiltonian structure. We will construct a stable Hamiltonian structure with a very precise property, which will ensure that the vorticity field corresponds to the Hamiltonian vector field of the additional first integral.

For some rescaling of the solutions (with $B$ of Morse-Bott type and eventually analytic) we can explicitely construct a one-form $\beta$ satisfying the condition $\iota_X d\beta=0$ and $\beta(X)>0$ and that it vanishes when evaluated on the curl of $X$.

\begin{lemma}\label{beta}
In Theorem \ref{ArnTop} denote by $Y$ the curl (for the constructed metric) of the steady flow $X$. Then, there exist a one-form $\beta$ such that
\begin{itemize}
\item $\beta(X)>0$,
\item $\iota_{X} d\beta=0$,
\item $\beta(Y)=0$.
\end{itemize}
\end{lemma}

\begin{proof}
Let us keep the simplifying assumption that the marked molecule has only simple atoms i.e. blocks of type $I-V$. We have a solution as constructed in Theorem \ref{ArnBl}: a given Arnold flow in each block using Proposition \ref{ArnBl}, and interpolations in each gluing locus using Lemma \ref{int}.

For a type $I$ block of the form $S^1\times D^2$ with coordinates $(\theta,r,\varphi)$, the preserved volume form is $\mu=rdr\wedge d\varphi\wedge d\theta$. For the type $II$ block, the volume form is $dt\wedge d\varphi\wedge d\theta$ for coordinates in $T^2 \times [1,2]$. In type $III$ and $IV$, the volume form is $\mu_N\wedge d\theta$, where $\mu_N$ is induced by an area form in the disk with two holes and $\theta$ a coordinate of the mapping torus. Finally, type $V$ block has as volume form $\mu_s\wedge d\theta$ where $\mu_S$ is induced by an area form in $S^1 \times [-1,1]$. We want to study the curl $Y$ of the solutions constructed in Proposition \ref{ArnBl}. In cases $I$ and $II$, the curl is of the form $Y=H(r) \pp{}{\varphi}$. In the three remaining cases, the curl equation writes
\begin{align}
\iota_Y \mu &= d\alpha \\
 			&= df \wedge d\theta,\label{base}
\end{align} 
where $f$ is some Morse-Bott function that coincides with the Bernoulli function up to a constant. We shall prove that the form $d\theta$ satisfies $d\theta(Y)=0$. Indeed, recall that $f$ is a first integral of $Y$. Assume that there is a point, hence an open subset $U$ of the block, where $d\theta(Y)\neq 0$. Using equation \eqref{base}, we deduce that
 $$\iota_Y \iota_Y \mu=\iota_Y df\wedge d\theta= -d\theta(Y)df,$$
and since $f$ is Morse-Bott there are regular points in $U$ where $\iota_Y\iota_Y\mu\neq 0$, reaching a contradiction . In conclusion, the one-form $\beta=d\theta$ satisfies $\beta(X)>0$, $\iota_X d\beta=0$ and $\beta(Y)=0$. \\

Finally, we only need to extend $\beta$ in the interpolation areas. In any of the applications of the interpolation lemma we did in Theorem \ref{ArnTop}, observe that in the boundary $U=[1,2]\times T^2$ we have $\beta|_{t=1}=d\theta$ and $\beta|_{t=2}= Cd\theta +Dd\varphi$. Let us prove that in an arbitrary interpolation, we can find a $\beta$ satisfying the boundary conditions and the required conditions inside. For the volume form $dt\wedge d\theta \wedge d\varphi$, which is preserved by $X$ and extends as a globally preserved volume form, we can compute the curl of $X$. It satisfies the condition $\iota_Y \mu=d(t\alpha)$. By writing such an equation in every interval (1)-(7) of Lemma \ref{int}, we find the following expression of $Y$.

\begin{equation*}
Y=
\begin{cases}
\pp{}{\varphi} - [tH_1'+H_1] \pp{}{\theta} \qquad &t\in I_1 \\
\pp{}{\varphi} - \pp{}{\theta} \qquad &t\in I_2\\
\pp{}{\varphi} - \pp{}{\theta} \qquad &t\in I_3\\
[1+H_4(C-1)+tH_4'(C-1)]\pp{}{\varphi} - \pp{}{\theta}\qquad &t\in I_4\\
C\pp{}{\varphi} - \pp{}{\theta} \qquad &t\in I_5\\
C\pp{}{\varphi} - [1+H_6(D-1)+tH_6'(D-1)]\pp{}{\theta} \qquad &t\in I_6\\
C\pp{}{\varphi}- D\pp{}{\theta}\qquad &t\in I_7
\end{cases}
\end{equation*}
Hence, we define $\beta$ as the following one-form.
\begin{equation*}
\beta= 
\begin{cases}
 d\theta + [tH_1'+H_1] d\varphi \qquad &t\in I_1 \\
d\theta + d\varphi \qquad &t\in I_2\\
d\theta + d\varphi \qquad &t\in I_3\\
[1+H_4(C-1)+tH_4'(C-1)]d\theta + d\varphi \qquad &t\in I_4\\
Cd\theta+d\varphi \qquad &t\in I_5\\
Cd\theta + [1+H_6(D-1)+tH_6'(D-1)]d\varphi \qquad &t\in I_6\\
Cd\theta + Dd\varphi\qquad &t\in I_7
\end{cases}
\end{equation*}
Such one-form clearly satisfies $\beta(Y)=0$. Furthermore, looking at the expression of $X$ in Lemma \ref{int}, we have $\beta(X)>0$ and $\iota_Xd\beta=0$.
\end{proof}

 The computations in the proof of the interpolation Lemma \ref{int} were adjusted so that one can easily use them to construct $\beta$ above.\\

Let $M$ be a graph manifold and $X$ an Arnold flow constructed as in Theorem \ref{main}, so that for some metric $g$ we have 
\begin{equation*}
\begin{cases}
\iota_X d\alpha&=-dB \\
d\iota_X\mu&=0
\end{cases}
\end{equation*}
where $\alpha=g(X,\cdot)$ and $\mu$ is the Riemannian volume. Denote by $Y$ the curl of $X$ with respect to $g$. By Lemma \ref{beta} we know that there is a one-form $\beta$ such that $\beta(X)>0$, $\beta(Y)=0$ and $\iota_X d\beta=0$. 

\begin{remark}
Note that for this $\beta$, the vector field $X$ is a rescaling of the Reeb field of the stable Hamiltonian structure $(\beta, \iota_X\mu)$. In particular, this shows that any ``admissible" Morse-Bott integral can be realized as the integral of a stable Hamiltonian structure.
\end{remark}

Consider in $M\times \R$, with coordinate $t$ in the second component, equipped with the two form
$$ \omega= dt \wedge \beta +td\beta + \iota_X \mu. $$
For $t$ small enough, it is a symplectic form. This is in fact the symplectization (cf. \cite{CV2}) of the stable Hamiltonian structure $(\beta,\iota_X\mu)$. The Euler flow $X$ and its curl can be seen as some Hamiltonian system with Bott integral in the symplectic manifold $(M\times [-\varepsilon,\varepsilon],\omega)$ .

\begin{proposition}\label{sym}
The pair $F=(t,-B)$ defines an integrable system in $M\times [-\varepsilon,\varepsilon]$. The Hamiltonian vector fields of $t$ and $B$ restricted to $M\times \{0 \}$ are respectively $X$ and its curl. 
\end{proposition}

\begin{proof}

The vector field $X$ satisfies that
$$ \iota_X \omega= -dt+t\iota_X d\beta= -dt,$$
which implies that $X$ is the Hamiltonian vector field of the function $H=t$. Furthermore, contracting $Y$ with the symplectic form we obtain
$$ \iota_Y \omega= \iota_Y \iota_X \mu + t\iota_Yd\beta. $$
Recall that $Y$ satisfies that $\iota_Y\mu=d\alpha$, so we have $\iota_Y\iota_X \mu= -\iota_X d\alpha=dB$. If $X_B$ denotes the Hamiltonian vector field of the function $-B$, we have $X_B|_{t=0}=Y$.

It remains to check that $F=(t,-B)$ defines an integrable system. Clearly $dt\wedge dB \neq 0$ almost everywhere, since $dB$ vanishes only on zero measure stratified sets of positive codimension. Furthermore, we have
\begin{align*}
\omega(X_t,X_{B})&=-\omega(X_B,X_t)\\
				&= -\iota_{X_B}\omega (X_t) \\
				&= -dB(X) \\
				&=0.
\end{align*}
The last equality follows from the first Euler equation: the fact that $\iota_X d\alpha=-dB$. Hence the two vector fields commute with respect to the symplectic form and $(t,-B)$ defines an integrable system.
\end{proof}

We obtain an alternative proof that any topological configuration of a Bott integrable system can be realized by some integrable system, with the additional property that the Hamiltonian vector field is, up to rescaling, the Reeb field of a stable Hamiltonian structure (Theorem \ref{maincor}). The realization theorem for Bott integrable systems was originally proved by Bolsinov-Fomenko-Matveev \cite{BFM}.

Proposition \ref{sym} unveils an example of an explicit (and expected) relation between Arnold's structure theorem and the classical Arnold-Liouville theorem in the theory of integrable systems. However the symplectization procedure to obtain integrable systems is an \emph{ad hoc} construction. In general, for a non-vanishing flow with an analytic or even Morse-Bott Bernoulli function, it is not possible to find a one-form as in Lemma \ref{beta}. 

In a point of the critical set of the Bernoulli function, we have $\iota_Xd\alpha=0$ and $\iota_Yd\alpha=0$. This implies that either $d\alpha$ vanishes and so does $Y$, or $Y$ is non-vanishing and parallel to $X$. It is clear that in the second case one cannot find a one-form such that $\beta(X)>0$ and $\beta(Y)=0$. It is possible to find examples where this happens, using Example 4.4 in \cite{KKP}.

\begin{example}
Consider the three torus $T^3$ with the standard metric on it $g=d\theta_1^2+ d\theta_2^2 +d\theta_3^2$. We take the volume preserving vector field 
$$ X= \sin^2 \theta_3 \pp{}{\theta_1} + \cos \theta_3 \pp{}{\theta_2}, $$
which is tangent to the tori obtained by fixing the third coordinate. The curl of $X$ is given by $Y=\sin \theta_3 \pp{}{\theta_1} + 2\sin \theta_3 \cos \theta_3 \pp{}{\theta_2}$. The dual form to $X$ is $\alpha= \sin^2 \theta_3 d\theta_1 + \cos \theta_3 d\theta_2$, from which we can deduce that the analytic Bernoulli function is $B=\frac{1}{2}(\sin^4\theta_3 + \cos^2\theta_3)$. Along the torus $\theta_3=\pi/2$, the derivative of the Bernoulli function vanishes. However, both $X$ and $Y$ are non-vanishing and parallel. Hence in such example one cannot find a one-form as in Lemma \ref{beta}.
\end{example}

We can also construct an example with a Morse-Bott Bernoulli function.

\begin{example}
Consider the solid torus as in the block of Section \ref{ssec:blocks}. Take coordinates $(\theta,x,y)$ in $S^1\times D^2$, and denote by $(r,\varphi)$ polar coordinates in $D^2$. We consider the one-form
$$ \alpha= (r^2+\varepsilon)d\theta + \varphi(r)xdy, $$
where $r=x^2+y^2$, the function $\varphi(r)$ is constantly equal to $1$ close to $0$ and equal to $0$ for $r\geq \delta$. The vector field will still be $X=\pp{}{\theta}$ and the volume form is $\mu=rdr\wedge d\varphi \wedge d\theta=dx\wedge dy \wedge d\theta$. We have
$$ d\alpha= 2rdr\wedge d\theta + (\pp{\varphi}{x}2x^3+\varphi)dx\wedge dy. $$
As before, we have $\iota_Xd\alpha=d(r^2)$, so the Bernoulli function is Morse-Bott $B=r^2$. However, constructing a metric with Lemma \ref{eul}, the curl of $X$ is no longer $\pp{}{\varphi}$. For $r>\delta$, we have $Y=\pp{}{\varphi}$. For $r$ very close to $0$ we have $\varphi(r)=1$ and hence $d\alpha=2rdr\wedge d\theta + dx\wedge dy$.  For such a form, the curl of $X$ is 
$$ Y=2\pp{}{\varphi} + \pp{}{\theta}, $$
which does not vanish at $r=0$. The construction in the solid torus can be extend to a compact manifold using our constructions in Section \ref{sec:arn}.
\end{example}

In both cases the one-form $\beta$ cannot be constructed, and in fact these steady Euler flows cannot be seen as integrable systems, as long as we ask the natural compatibility conditions that $X$ and $Y$ are respectively the Hamiltonian vector fields of the integrals $t$ and $B$ (or $-B$). Indeed if $Y$ was the Hamiltonian vector field of $B$, it should always vanish at the critical points of $B$, since it would be defined by the equation $\iota_Y \omega=-dB$ for some symplectic form $\omega$.

\section{Singular Morse-Bott Arnold flows}

In this last section, we study the most general case of Euler flows with a Morse-Bott Bernoulli function. Those are singular Arnold flows, and by singular we mean that we allow the vector field $X$ to have stagnation points. We will prove that these steady solutions do not exist in non-graph manifolds, generalizing Theorem \ref{top} in the Morse-Bott case by dropping the non-vanishing assumption.

\subsection{Critical sets of the Bernoulli function}

Let us first analyze the level sets of a Morse-Bott Bernoulli function. A first lemma is the non-existence of non-degenerate critical points.

\begin{lemma}\label{lem:nosad}
Let $X$ be a steady Euler flow with smooth Bernoulli function $B$. Then $B$ does not have any non-degenerate critical point.
\end{lemma}  
\begin{proof}
Assume there is a non-degenerate critical point $p$. A simple argument that we describe now rules out the case of a maximum or a minimum \cite{P}. By the Morse lemma there is a local chart $(U,(x,y,z))$ around $c$ such that the function (up to constant) is $B=x^2+y^2+z^2$ or $B=-x^2-y^2-z^2$. But then either $B^{-1}(\varepsilon)$ or $B^{-1}(-\varepsilon)$ is a regular level set diffeomorphic to a sphere. This is a contradiction with Arnold's theorem, which ensures that all regular level sets are tori. It only remains the case of a saddle point. \\

Let $p$ be a saddle point of $B$. Again by the Morse lemma, there are coordinates $(x,y,z)$ such that up to constant we have
$$ B=x^2+y^2-z^2. $$
In these coordinates, we can write 
$$X=X_1\pp{}{x}+X_2\pp{}{y}+X_3\pp{}{z}$$
and 
$$ d\alpha=ady\wedge dz + bdx\wedge dz + cdx\wedge dy, $$
where $X_1,X_2,X_3,a,b,c$ are functions depending on $x,y,z$. The first Euler equation $\iota_Xd\alpha=-dB$ implies
\begin{equation}\label{eq:system}
\begin{cases}
X_1b-X_2a=2z\\
X_1c-X_3a=-2y\\
-X_3b-X_2c=-2x
\end{cases}
\end{equation}

We claim that $d\alpha|_p\neq 0$. Assume the converse, that $d\alpha|_p=0$. Because $B$ is an integral of $X$, the vector field $X$ vanishes at $p$ necessarily. There are several ways to see this. In the coordinates $(x,y,z)$, denote by $B_1,B_2,B_3$ the derivatives of $B$ with respect to $x,y$ and $z$. The fact that $B$ is an integral of $X$ implies that $X_iB_i=0$. Deriving this equation and restricting to the critical point we get
$$(X_1,X_2,X_3) .D^2B (X_1,X_2,X_3)^T|_p=0, $$
where $D^2B$ denotes the Hessian matrix of $B$. This matrix is non-degenerate at $p$ and hence we deduce that $(X_1,X_2,X_3)|_p=0$.

Taking the Taylor expansion of the functions $X_1,X_2,X_3$ and $a,b,c$, they all have a vanishing coefficient of order $0$. In particular, the combinations of the system \eqref{eq:system} yield functions that vanish at least for orders $0$ and $1$. This contradicts the system of equations, since the right side does not vanish at order one.\\

We deduce that $d\alpha|_p\neq 0$. However, we know that the vorticity $Y$ is determined by the equation
$$ \iota_Y \mu=d\alpha,  $$
which implies that $Y|_p\neq 0$. But $B$ is also an integral of $Y$ because $\iota_YdB=-\iota_Y\iota_Xd\alpha=\iota_X \iota_Yd\alpha=0$ and by the previous argument this implies that $Y|_p=0$. We obtain a contradiction, and conclude that a saddle point cannot exist.
\end{proof}

Observe that the previous lemma applies without any assumption on the ambient space or the metric, so any stationary solution to the Euler equations in flat spaces like the Euclidean space $\mathbb{R}^3$ or the flat torus $T^3$ does not admit a Bernoulli function with a non-degenerate critical point. When the ambient manifold is not of graph type, this lemma is enough to prove the non-existence of Bott integrable steady Euler flows. Let us further give a topological characterization of the critical level set of a Morse-Bott Bernoulli function.

\begin{lemma}\label{lem:stru}
Let $X$ be an Euler flow with Morse-Bott Bernoulli function. Let $c$ be a critical submanifold. Then $c$ is either a circle, a torus, or a Klein bottle. If $c$ is a circle of saddle type, denote by $Z$ a regular component of the critical level set containing $c$. Then $Z$ is an orientable finitely punctured surface with finite genus.
\end{lemma}

\begin{proof}
By Lemma \ref{lem:nosad}, each critical submanifold $c$ is of dimension one or two. If $c$ is two-dimensional, it has to be a compact surface. Furthermore, the regular level sets in a trivial neighborhood of $c$ must be tori because of Arnold's theorem. This implies that $c$ is either a torus or a Klein bottle. If $c$ is one-dimensional, it is compact and hence a circle.

For the second part of the lemma, denote by $Z$ a $2$-dimensional strata of a critical level set: it is an open embedded surface. The fact that $Z$ is orientable follows from the fact that $dB\neq0$ everywhere in $Z$ and is transverse to it. Then the gradient of $B$, which satisfies $g(\operatorname{grad}B,\cdot)=dB$, is a vector field everywhere transverse to $Z$. This implies that $Z$ is an open, orientable surface. By compactness, it has a finite amount of punctures (approaching the critical circles) and has finite genus.
\end{proof}

\subsection{Non-existence of Bott integrable steady Euler flows}

We proceed with the proof of Theorem \ref{thm:sing}.

\begin{proof}[Proof of Theorem \ref{thm:sing}]

We will first show that a stratified Bernoulli function has necessarily a non-empty $0$-strata if $M$ is not of graph type. A function is stratified \cite{EG0} if its critical values are isolated and the critical level sets are Whitney stratified sets of codimension greater than zero. This includes both analytic and Morse-Bott functions. The claim follows easily from the theory of tame functions introduced in \cite{FM}.

Assume that there are no $0$-strata. The $1$-strata are necessarily critical circles, by compactness. By Arnold's theorem, every regular level set is a torus. Then the function $B$ is a tame function in the sense of \cite{FM}, and by \cite[Theorem 3]{FM} $M$ has to be a graph manifold: this is a contradiction.

Hence, if $B$ is a Morse-Bott Bernoulli function of some steady Euler flow on a non-graph manifold, it necessarily has a non-degenerate critical point. By Lemma \ref{lem:nosad}, this is not possible and we conclude that such steady flow cannot exist.
\end{proof}
The Morse-Bott assumption was key in the proof. In the general case of an analytic Bernoulli function, the problem of the existence of integrable steady fluids in non-graph manifolds remains open.
\appendix

\section{$3$-atoms and topological classification}

In this appendix, we introduce the notion of $3$-atom as in \cite{F1}, show how to construct an Arnold flow on an arbitrary $3$-atom and discuss the topological classification of the moduli of Morse-Bott Arnold flows.

Given a non-vanishing vector field with a Morse-Bott integral $F$, we denote by $L$ a critical level set of $F$. We are now in the general case and a single critical level set can have more than one critical circle. An example is given by the level set of the height function in Figure \ref{fig:emb}, where the cylinders of the boundary components of $\Sigma_0$ merge. The level set in the total space $\Sigma_0 \times S^1$ is Figure \ref{fig:double8} times a circle.

\begin{figure}[!ht] 

\begin{center}
\begin{tikzpicture}
     \node[anchor=south west,inner sep=0] at (0,0) {\includegraphics[scale=0.16]{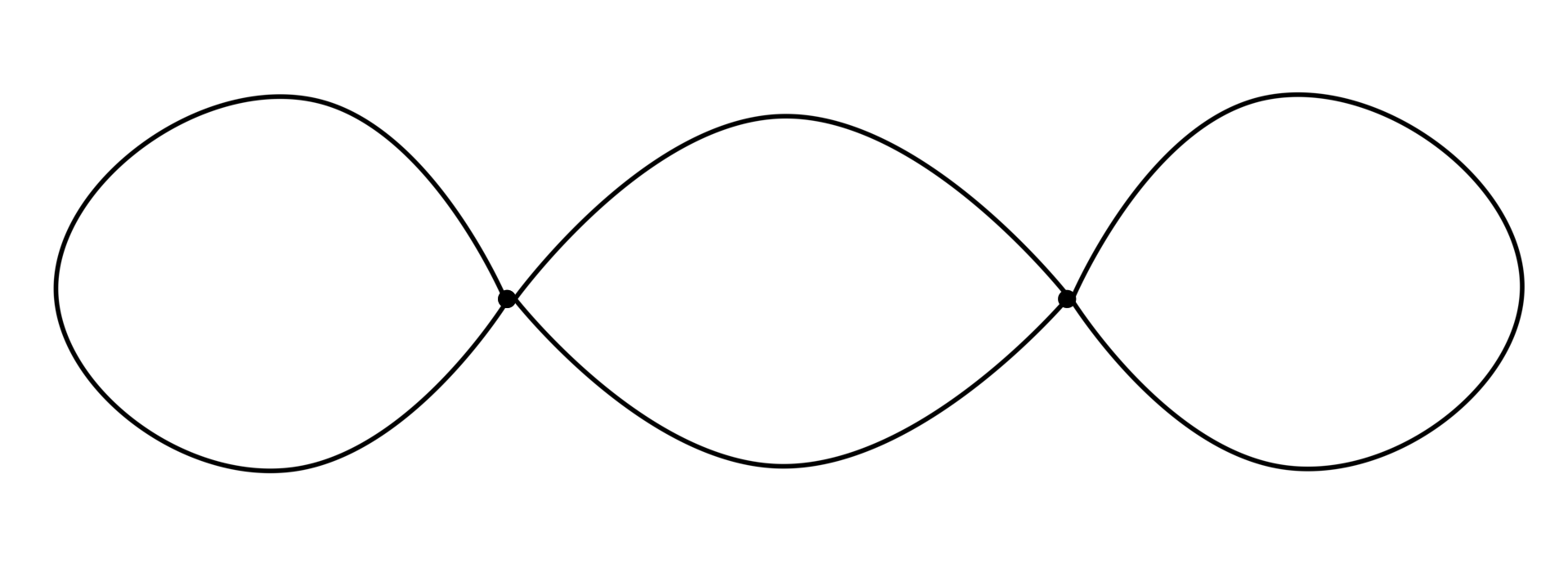}};
     
%     \node at (4.5,5){$\Sigma_0$};
%     \node at (0,0)[scale=0.8]{$\partial D_1$};
%     \node at (2,0)[scale=0.8]{$\partial D_2$};
%     \node at (4,0)[scale=0.8]{$\partial \tilde D_1$};
%     \draw [->] (6,0) -- (6,6);
%     \node at (6.5,3) {$h$};
%     \node at (6.35,0.43) {$1$};
%     \node at (6,0.43) {$-$};
\end{tikzpicture}
\end{center}
\caption{Non simple $2$-atom}
\label{fig:double8}
\end{figure}

We consider a neighborhood $U(L)$ of $L$ foliated by the function: that is $F^{-1}([c-\varepsilon,c+\varepsilon])$ where $f(L)=c$. We call the space $F^{-1}([c-\varepsilon,c+\varepsilon])$ together with the foliation by level sets a three-atom. Originally, these are considered up to diffeomorphism preserving the foliation and the orientation induced by the flow in the possibly existing critical circles.

It turns out that the topological classification of three-atoms depends on the classification of two-atoms. A two-atom is the neighborhood of a singular level set of a Morse function on a surface. That is, again, $U(L') =f^{-1}([c-\varepsilon,c+\varepsilon])$ where $f$ is a Morse function in a surface and $L'$ a critical level set of $f$. The classification of three-atoms is then the following. A three-atom is always of the form $P^2 \times S^1$ or $P^2 \tilde \times S^1$. Here $P^2$ denotes some two-atom, and the second case is a twisted product that denotes the mapping torus by certain involution $\tau : P^2 \rightarrow P^2$ which preserves the Morse function $f$ inducing the foliation in $P^2$. It follows from the description of an arbitrary three-atom that $f$ glues well with the image of $f$ by the gluing diffeomorphism, and yields a well-defined Morse-Bott function $\tilde f$. The blocks $I, III$ and $IV$ presented in \ref{ssec:blocks2} are the three-atoms in the case where the Bernoulli function only has a single critical circle in the critical level set. Blocks $II$ and $V$ are introduced to take into account the case of critical surfaces.

One can construct, similarly to type $III$ and $IV$ blocks, an Arnold flow in a given $3$-atom using its structure of mapping torus. If we denote by $\theta$ the coordinate in the $S^1$ component, we take as vector field $X=\pp{}{\theta}$. As one-form we take $\alpha= (K+\tilde f) d\theta$, where $K$ is a constant such that $\alpha(X)>0$ everywhere. Finally, take $B=C+\tilde f$ as Bernoulli function for some other constant $C$. We clearly have that $\iota_X d\alpha=-dB$. Given any area form $\omega$ in $P^2$, the area form $\omega + \tau^*\omega$ is invariant by the mapping torus and hence $X$ is volume-preserving for some volume. Lemma \ref{eul} concludes that it is an Arnold flow. The torus boundary components are regular level sets of the Bernoulli function. Hence, one can apply the arguments of the proof of Theorem \ref{ArnTop} that we used for simple atoms in this more general setting. Instead of a graph whose vertices are blocks of type $I-V$, one can have blocks of type $II,III$ and any other possible $3$-atom. It is also immediate to check that the proof of Theorem \ref{maincor} also applies for Morse-Bott function with atoms of arbitrary complexity.  The one-form $\beta$ in Lemma \ref{beta} can be constructed in a given $3$-atom analogously to how it is done for blocks of type $III$ and $IV$. In \cite{F1}, the study of equivalence classes of such more general graphs gives rise to the notion of a marked molecule. Marked molecules classify topologically stable Bott integrable systems. In our setting, we were just interested in the topology of $B$, i.e. the foliation by level sets, and not in the orientations at the critical circles. When we forget about the orientation of the critical circles and drop the topologically stable condition, the classification is also possible in terms of equivalence classes of these graphs (molecules with gluing coefficients). In that case, however, it becomes more technical that with the simplifying assumptions taken in \cite{F1}.

 If we keep track of the orientation of the critical circles and take the simplifying assumption that the orientation induced by the fluid on the critical circles is compatible with that orientation in each critical level set, then the marked molecule is a complete topological invariant of Morse-Bott Arnold flows. 
 
\begin{corollary}
Marked molecules classify topologically the moduli of non-vanishing Euler flows with Morse-Bott Bernoulli function.
\end{corollary}
This classification can be compared to \cite{IK}, where vorticity functions of Morse type are topologically classified in the context of the Euler equations in surfaces.\\

%The acknowledgments section should not be numbered.
\section*{Acknowledgments} The author acknowledges financial support from the Spanish Ministry of Economy and Competitiveness, through the Mar\'ia de Maeztu Programme for Units of Excellence in R\&D (MDM-2014-0445) via an FPI grant. The author is also supported by the AEI grant PID2019-103849GB-I00/ AEI/ 10.13039/ 501100011033, and AGAUR grant 2017SGR932. The author is grateful to Eva Miranda, C\'edric Oms, and Daniel Peralta-Salas for useful conversations, as well as to Alexey Bolsinov for helpful correspondence concerning the topological classification of Bott integrable systems.

%%%%%%%%%%%%%%%%%%%%%%%%%%%%%%%%%%%%%%%%%%%%%%%%%%%%%%
%          7. REFERENCES SECTION
%%%%%%%%%%%%%%%%%%%%%%%%%%%%%%%%%%%%%%%%%%%%%%%%%%%%%%

%       READ THIS SECTION CAREFULLY

% Each of the references below MUST be cited in your article above. Do not include references that are not cited in your article.

% Follow the examples below carefully. We strongly suggest that you copy and paste your reference information directly into our examples.

% List all references in alphabetical order according to the first author’s last name.

% Verify each URL works correctly and can be accessed properly. Your URL links should be to reputable websites. The command line for a website link begins with: \url{ }

% Do not add MR or DOI numbers to your references. AIMS production staff will add this information.

% Using BibTex is not recommended but can be handled.

\end{document}